\documentclass[10pt, epsfig]{article}
\usepackage{epsfig, amsmath, amssymb, amsthm, times}
\def\N{\mathbb{N}}\def\Z{\mathbb{Z}}\def\E{\mathbb{E}}\def\R{\mathbb{R}}
\def\d{\frac{1}{2}\, }
\newcommand{\bel}{\begin{equation}\label}
\newcommand{\ee}{\end{equation}}

\def\R{{\mathbb R}}
\def\C{{\mathbb C}}
\def\N{{\mathbb N}}
\def\Z{{\mathbb Z}}

\def\tr{\;\mathrm{tr}\,}

\def\<{\langle}
\def\>{\rangle}
\def\P{{\mathbb P}}

\def\E{{\mathbb E}}
\def\eps{\epsilon}

\def\bel{\begin{equation}\label}
\def\ee{\end{equation}}

\newtheorem{theorem}{Theorem}[section]
\newtheorem{proposition}[theorem]{Proposition}

\author{G\'erard Letac\thanks {Institut de Math\'ematiques, Universit\'e de Toulouse, 31062 Toulouse, France. email \texttt{gerard.letac@math.univ-toulouse.fr}} and Jacek Weso\l owski\thanks{Wydzia{\l} Matematyki i Nauk Informacyjnych, Politechnika Warszawska, Warszawa, Poland, e-mail: \texttt{wesolo@mini.pw.edu.pl}}}

\date{\today}
\title{About an extension of the Matsumoto-Yor property}

\begin{document}  \maketitle
	\begin{abstract}
		If $\alpha,\beta>0$ are distinct and if $A$ and $B$ are independent non-degenerate positive random variables such that
		$$S=\tfrac{1}{B}\,\tfrac{\beta A+B}{\alpha A+B}\quad \mbox{and}\quad T=\tfrac{1}{A}\,\tfrac{\beta A+B}{\alpha A+B}
		$$
		are independent, we prove that this happens if and only if the $A$ and $B$ have generalized inverse Gaussian distributions with suitable parameters. Essentially, this has already been proved in Bao and Noack (2021) with supplementary hypothesis on existence of smooth densities. 
		
		The sources of these questions are an observation about independence properties of the exponential Brownian motion due to Matsumoto and Yor (2001) and a recent work of Croydon and Sasada (2000) on random recursion models rooted in the discrete Korteweg - de Vries equation, where the above result was conjectured.
		
		We also extend the direct result to random matrices proving that a matrix variate  analogue of the above independence property is satisfied by independent matrix-variate GIG  variables. The question of characterization of GIG random matrices through this independence property remains open.   
	\end{abstract}
	

\section{Detailed balance equation for discrete KdV model}
	
	Croydon and Sasada (2020) introduced two so called detailed balance equations. For type I model we say that a pair of probability measures  $\mu$  and $\nu$ on some spaces $\mathcal U$ and $\mathcal V$ satisfies the detailed balance equation for a map $F=(F_1,F_2):\mathcal U\times \mathcal V\to \mathcal U\times \mathcal V$ iff 
	\bel{dbe}
	F(\mu\otimes\nu)=\mu\otimes\nu,
	\ee
	where $F(\mu\otimes\nu)$ means $(\mu\otimes\nu)\circ F^{-1}$. 
	
For type II model we say that two pairs of probability measures  $\mu$, $\nu$ on  $\mathcal U$, $\mathcal V$ 	and $\widetilde\mu$, $\widetilde\nu$ on $\widetilde{\mathcal U}$,  $\widetilde{\mathcal V}$ satisfy the detailed balance equation for a map $F=(F_1,F_2):\mathcal U\times \mathcal V\to \widetilde{\mathcal U}\times\widetilde{\mathcal V} $ iff 
	\bel{dbII}F(\mu\otimes\nu)=\widetilde\mu\otimes\widetilde\nu.\ee
	In other terms for the type II model, one considers independent random variables $X$ and $Y$ with distributions $\mu$ and $\nu$ such that  random variables $U=F_1(X,Y)$ and $V=F_2(X,Y)$  are independent and have distributions $\widetilde\mu$ and $\widetilde{\nu}$, respectively.  Type I model is more restrictive since it imposes the additional condition that $X\stackrel{d}{=}U$ and $Y\stackrel{d}{=}V$.  
	
	As an example of type II model, if $\mathcal U=\mathcal V=\widetilde{\mathcal U}=\widetilde{\mathcal V}=(0,\infty)$ and 
$F(x,y)=(x+y,\frac{1}{x}-\frac{1}{x+y})$ Matsumoto and Yor (2001) have observed that if $X$ and $Y$ are respectively Gamma and generalized inverse Gaussian (GIG)  distributed with suitable parameters and independent, the same is true for the pair $U,V$ with different parameters (a description of the GIG distribution and of its limiting cases appears  in Section 2 and in Section 3 respectively). This phenomena has been called the Matsumoto-Yor (MY) property, a name  coined by Stirzaker (2005), p. 43. In Letac and Weso\l owski (2000) a characterisation by independence of $X$, $Y$ and independence of $U$, $V$ of these pairs Gamma, GIG is given in .

In the present paper we investigate a type II model, again with $\mathcal U=\mathcal V=\widetilde{\mathcal U}=\widetilde{\mathcal V}=(0,\infty)$ which has been considered by Croydon and Sasada (2020), Sec. 3.2. They introduced the involutive map $F_{dK}^{(\alpha,\beta)}:(0,\infty)^2\to(0,\infty)^2$ defined by 
	\bel{fdk}
	F_{dK}^{(\alpha,\beta)}(x,y)=\left(y\tfrac{\beta xy+1}{\alpha xy+1},\,x\tfrac{\alpha xy+1}{\beta xy+1}\right)
	\ee
	for $\alpha,\beta\ge 0$ and $\alpha\neq 
\beta$ in connection with the modified discrete  Korteweg - de Vries (mKdV) model. The discrete mKdV model  can be  described by the following dynamics: to each point $(n,t)\in \Z^2$ one associates a vector $(x_n^t,y_n^t)\in \R^2$ defined by $(x_n^{t-1},y_{n-1}^t)$  and by the formula 
$$(x_n^t,y_n^t)=F_{dK}^{(\alpha,\beta)}(x_n^{t-1},y_{n-1}^t).$$ For instance $(x_n^t,y_n^t)$ is known for all $n,t>0$ if it is known along $(0,n)$
and $(t,0)$ for all $n,t\geq 0.$

If $X$ and $Y$ are independent as well as $U$ and $V$ defined by $(U,V)=F_{dK}^{(\alpha, \beta)}(X,Y)$ the fact that $F_{dK}^{(\alpha, \beta)}$ is involutive permits to create a stationary measure on $\Z^2$ for the above dynamics  by taking $ (x_n^t,y_n^t)=(X,Y)$ or $(U,V)$ according to the fact that $t+n$ is even or odd. The straight discrete KdV model corresponds to the particular case $\beta=0.$
For details, see Sect. 2.2 in Croydon, Sasada and Tsujimoto (2020) with slightly different notations. 
	
	Croydon and Sasada (2020) observed that the choice of $\mu$ and $\nu$ as generalized inverse Gaussian (GIG) distributions  with suitable parameters fits with the detailed balance equation of type II model . Actually, their Proposition 3.9 considers the type I model, but the proof translates immediately to the type II model with $\widetilde\mu$ and $\widetilde{\nu}$ being also GIG distributions. Furthermore, their Conjecture 8.6 predicts that these sets of GIG distributions are the only possible ones for the type II model. This has been proved by Bao and Noack (2021) under the technical assumptions that the involved distributions have strictly positive densities on $(0,\infty)$ with second derivatives. Their approach was based on expressing the independence condition through the functional equation for logarithms $r_X$, $r_Y$, $r_U$ and $r_V$ of densities of $X$, $Y$, $U$ and $V$. They proved that the Jacobian of the change of variable $(u,v)=F_{dK}^{(\alpha,\beta}(x,y)$ is  $-1$. Consequently, they identified  the functional equation to be solved as
	$$
	r_U\left(\tfrac{y(1+\beta xy)}{1+\alpha xy}\right)+r_V\left(\tfrac{x(1+\alpha xy)}{1+\beta xy}\right)=r_X(x)+r_Y(y),\qquad x,y>0.
	$$
	
	The first main result of the present paper proves the Croydon and Sasada Conjecture 8.6 in full generality. The technique uses a extended Laplace transform $L_X$ of a positive random variable $X$ defined by  
	\bel{lapl}L_X(s,\sigma,\theta)=\E\,X^s\,e^{\sigma X+\frac{\theta}{X}},\qquad(s,\sigma,\theta)\in\R\times(0,\infty)^2
	\ee and leads to a related function $g_s$ which satisfies the classical Bessel differential equation, see e.g. Watson (1966), Sec. 3.7,
	$$
	z^2g_s''(z)+zg_s'(z)-(z^2+\nu^2)\,g_s(z)=0.
	$$ One can mention here that the main tool of the paper  Letac and Weso\l owski (2000) is the simpler function $(\sigma,\theta)\mapsto \E\,  e^{\sigma X+\frac{\theta}{X}}.$
	
	To describe our second main result we note that the limiting cases for $\alpha=0$ or $\beta=0$ have been considered in Letac and Weso\l owski (2000). In particular, that paper deals with the natural extension from the domain $(0,\infty)$ to the cone of positive definite matrices of order $r$, which we denote by $\Omega_+$. Following the matrix variate context, here, for $\alpha,\beta>0$ we prove that matrix variate GIG distributions, which we rather denote MGIG, satisfy the detailed balance equation for type II model with the function $F_{dK}^{(\alpha,\beta)}:\Omega_+^2\to\Omega_+^2$ defined by
	$$
	F_{dK}^{(\alpha,\beta)}(x,y)=\left(y(I+\alpha xy)^{-1}(I+\beta xy),\,x(I+\beta yx)^{-1}(I+\alpha yx)\right),
	$$
	where $x$ and $y$ are positive definite matrices of order $r$ and $I$ is the identity matrix.
	
	Section 2 describes the GIG distribution and states our  first main result in Theorem \ref{Th:psiAB}. Section 3 examines the limiting cases of $\alpha$ or $\beta$ equal zero. The long Section 4 gives the proof of Theorem \ref{Th:psiAB}.  Section 5 considers the symmetric matrices generalization and gives our second main result. Section 6 comments on references about the MY property.

\section{GIG laws and the characterization}
	Denote by $\mathrm{GIG}(\lambda,a,b)$, $\lambda\in\R$, $a,b>0$, the generalized inverse Gaussian distribution. It is defined by  the density
	$$
	f(x)=\left(\tfrac{a}{b}\right)^{\lambda/2}\,\tfrac{1}{2K_\lambda(2\sqrt{ab})}\, x^{\lambda-1}\,\exp\left(-ax-\tfrac{b}{x}\right)\,I_{(0,\infty)}(x),
	$$
	where $K_{\nu}$ is the modified Bessel function of the third kind, see \eqref{besse}. 
	
	The following reciprocity property of the GIG distribution is well known and easily seen: if a random variable $X$ has the $\mathrm{GIG}(\lambda,a,b)$ distribution, then the distribution of $X^{-1}$ is $\mathrm{GIG}(-\lambda,b,a)$. A  less known property of GIG is the form of the extended  Laplace transform defined in \eqref{lapl}, which has the form
	\begin{equation}
	\label{LX}
	L_X(s,\sigma,\theta)=\left(\tfrac{b-\theta}{a-\sigma}\right)^{\tfrac{\lambda+s}{2}}\,\left(\tfrac{a}{b}\right)^{\tfrac{\lambda}{2}}\,\tfrac{K_{\lambda+s}(2\sqrt{(a-\sigma)(b-\theta)})}{K_{\lambda}(2\sqrt{ab})}
		\end{equation}
		for $s\in\R$, $\sigma<a$ and $\theta<b$. 
		Note that $L_X(0,\sigma,\theta)$ as above with $s=0$ uniquely determines the law $X\sim\mathrm{GIG}(\lambda,a,b)$. 
	Actually $L(s,\theta,\sigma)$ of the above form for any fixed $s$ determines the GIG distribution as shown by the next proposition.
	\begin{proposition}
		Assume that there exists $s\in\R$ such that $$L_X(s,\theta,\sigma)=c\left(\tfrac{b-\sigma}{a-\theta}\right)^{\tfrac{\nu}{2}}\,K_\nu(2\sqrt{(a-\theta)(b-\sigma)}),\quad \theta<a,\;\sigma<b$$
		for some constants $a,b,c>0$. Then $X\sim\mathrm{GIG}(\nu-s,a,b)$.
	\end{proposition}
	\begin{proof}
		Take $\tilde a<a$ and $\tilde b<b$. Consider a random variable $Y$ with the distribution
		$$
		\P_Y(dx)=\tfrac{x^se^{\tilde a x+\frac{\tilde b}{x}}\,\P_X(dx)}{L_X(s,\tilde a,\tilde b)}.
		$$
		Then for $\theta<a-\tilde a$ and $\sigma<b-\tilde b$ we have
		$$
		L_Y(0,\theta,\sigma)=\tfrac{L_X(s,\tilde a+\theta,\tilde b+\sigma)}{L_X(s,\tilde a,\tilde b)}=\left(\tfrac{(b-\tilde b-\sigma)(a-\tilde a)}{(a-\tilde a-\theta)(b-\tilde b)}\right)^{\nu/2}\,\tfrac{K_\nu(2\sqrt{(a-\tilde a-\theta)(b-\tilde b-\sigma)})}{K_\nu(2\sqrt{(a-\tilde a)(b-\tilde b)})}.
		$$
		Comparing this form with \eqref{LX} we conclude that $Y\sim \mathrm{GIG}(\nu,a-\tilde{a},b-\tilde b)$. Consequently,
		$$
		\P_X(dx)\propto\tfrac{1}{x^se^{\tilde a x+\frac{\tilde b}{x}}}\,\,\P_Y(dx)=\left(\tfrac{b-\tilde b}{a-\tilde a}\right)^{\nu/2}\,\tfrac{x^{\nu-s-1}\,\exp\left(-ax-\tfrac{b}{x}\right)}{K_{\nu}(2\sqrt{(a-\tilde a)(b-\tilde b)})}.
		$$
		Thus, $\P_X(dx)\propto x^{\nu-s-1}\,\exp\left(-ax-\tfrac{b}{x}\right)$ and the result follows.
	\end{proof}
	
The result below gives the solution of the Croydon-Sasada conjecture for $\alpha,\beta>0$ and $F_{dK}^{(\alpha,\beta)}$  defined by \eqref{fdk}. As said before, it has  already been proved by Bao and Noack (2021) with extra hypothesis and  different techniques.
	
	\begin{theorem}\label{Th:FdK}
		Let $X$ and $Y$ be non-Dirac, non-negative independent random variables. Let $\alpha>0$ and $\beta>0$ be distinct real numbers. Let
		$$
		(U,V)=F_{dK}^{(\alpha,\beta)}(X,Y).
		$$
		If $U$ and $V$ are independent then there exist $c_1,c_2>0$ and $\lambda\in \R$ such that \bel{XY}
		X\sim\mathrm{GIG}(-\lambda,\alpha c_1,c_2)\quad\mbox{and}\quad Y\sim \mathrm{GIG}(-\lambda,\beta c_2,c_1)
		\ee 
		and 
		\bel{UV}
		U\sim\mathrm{GIG}(-\lambda,\alpha c_2, c_1)\quad \mbox{and}\quad V\sim \mathrm{GIG}(-\lambda,\beta c_1,c_2).
		\ee
		
	\end{theorem}
	
	We now prove that  Theorem \ref{Th:FdK} is equivalent to the following theorem.
	\begin{theorem}\label{Th:psiAB}
		Let $A$ and $B$ be non-Dirac, non-negative independent random variables. Let $\alpha>0$ and $\beta>0$ be distinct real numbers. Let	$$S=\tfrac{1}{B}\,\tfrac{\beta A+B}{\alpha A+B}\quad \mbox{and}\quad T=\tfrac{1}{A}\,\tfrac{\beta A+B}{\alpha A+B}
		$$
		If $S$ and $T$ are independent then there exist $c_1,c_2>0$ and $\lambda\in \R$  such that
		\bel{AB} A\sim\mathrm{GIG}(-\lambda,\alpha c_1,c_2)\quad\mbox{and}\quad B\sim \mathrm{GIG}(\lambda,c_1,\beta c_2)
		\ee
		and
		\bel{ST}S\sim \mathrm{GIG}(-\lambda,\alpha c_2, c_1)\quad\mbox{and}\quad T\sim\mathrm{GIG}(\lambda,c_2,\beta c_1).\ee
	\end{theorem}
	To see this equivalence it will be convenient to introduce a modification of $F_{dK}^{(\alpha,\beta)}$ of the form 
	$$\psi^{(\alpha,\beta)}=I_2^{-1}\circ F_{dK}^{(\alpha,\beta)}\circ I_2,$$
	where $I_2:(0,\infty)^2\to(0,\infty)^2$ is defined by $I_2(x,y)=\left(x,\,\tfrac{1}{y}\right)$. Since $F_{dK}^{(\alpha,\beta)}$ is involutive it is clear that $\psi^{(\alpha,\beta)}$ is also an involution on $(0,\infty)^2$. Note that
	\bel{psiab}
	\psi^{(\alpha,\beta)}(x,y)=\left(\tfrac{1}{y}\, \tfrac{\beta x+y}{\alpha x+y}, \tfrac{1}{x}\, \tfrac{\beta x+y}{\alpha x+y}\right)=\left(\tfrac{\beta}{\alpha y}+\tfrac{\alpha-\beta}{\alpha(\alpha x+y)},\,\tfrac{1}{x}-\tfrac{\alpha-\beta}{\alpha x+y}\right).
	\ee 
	For a probability measure $\nu$ of $(0,\infty)$ denote by $\nu^{(-1)}$ a measure defined by $\nu^{(-1)}(D)=\nu(1/D)$ for all Borel sets $D\subset(0,\infty)$ and $1/D=\{1/x,\;x\in D\}$. It is obvious that $\mu,\nu$ and $\widetilde\mu,\widetilde\nu$ satisfy the detailed balance equation \eqref{dbII} for $F_{dK}^{(\alpha,\beta)}$   iff  $\mu,\nu^{(-1)}$ and $\widetilde\mu,\widetilde\nu^{(-1)}$ satisfy  \eqref{dbII} for $\psi^{(\alpha,\beta)}$. Since 
	$$(S,1/T)=I_2(S,T)= F_{dK}^{(\alpha,\beta)}\circ I_2(A,B)=F_{dK}^{(\alpha,\beta)}(A,1/B)=(U,V)$$
	the equivalence  follows immediately from \eqref{XY}, \eqref{UV} and \eqref{AB}, \eqref{ST} in view of the reciprocity property of the GIG.

\section{The limiting cases of $\alpha=0$ or $\beta=0$ and the MY property}

Note that inserting $\beta=0$ in \eqref{fdk} or \eqref{psiab} gives
\bel{fdk0}
F_{dK}^{(\alpha,0)}(x,y)=\left(\tfrac{y}{\alpha xy+1},\,x\,(\alpha xy+1)\right)
\ee
and
\bel{psi0}
\psi^{(\alpha,0)}(x,y)=\left(\tfrac{1}{\alpha x+y},\,\tfrac{1}{x}-\tfrac{\alpha}{\alpha x+y}\right).
\ee
One can expect that the problem of detailed balance equation of type II model for $F_{dK}^{(\alpha,0)}$ or similarly for $F_{dK}^{(0,\beta)}$ can be derived by  taking the limit as $\alpha\to 0$ or $\beta\to 0$ in the Theorem \ref{Th:FdK} or in its equivalent Theorem \ref{Th:psiAB}. Indeed, for $\lambda>0$
$$
\mathrm{GIG}(\lambda,a,b)\stackrel{w}{\to} \mathrm{G}(\lambda,a)\quad \mbox{for }\;b\to 0,
$$
where $\stackrel{w}{\to}$ denotes the weak convergence of probability measures and $\mathrm{G}(\lambda,a)$ is the Gamma distribution defined by the density 
$$
f(x)=\tfrac{a^{\lambda}}{\Gamma(\lambda)}\,x^{\lambda-1}\,e^{-a x}\,I_{(0,\infty)}(x),
$$
Similarly, for $\lambda>0$
$$
\mathrm{GIG}(-\lambda,a,b)\stackrel{w}{\to} \mathrm{InvG}(\lambda,b)\quad \mbox{for }\;a\to 0,
$$
where $\mathrm{InvG}(\lambda,b)$ is the inverse Gamma distribution defined by the density
$$
f(x)=\tfrac{b^{-\lambda}}{\Gamma(\lambda)}\,x^{-\lambda-1}\,e^{-b/x}I_{(0,\infty)}(x).
$$ 
Note that without any loss of generality regarding the independence property we can assume that $\alpha=1$ in \eqref{fdk0} or \eqref{psi0}. Actually, we would rather consider $\psi^{(1,0)}$ which assumes the form 
$$
\psi^{(1,0)}(x,y)=\left(\tfrac{1}{x+y},\,\tfrac{1}{x}-\tfrac{1}{x+y}\right),
$$
since it is directly related to the MY property. More specifically, this property says:

Let $A\sim\mathrm{GIG}(-\lambda,c_1,c_2)$ and $B\sim \mathrm G(\lambda,c_1)$ be independent and $(S,T)=\psi^{(1,0)}(A,B)$. Then $S$ and $T$ are independent, $S\sim \mathrm{GIG}(-\lambda,c_2,c_1)$ and $T\sim \mathrm G(\lambda,c_2)$. Matsumuto and Yor (2001) discovered this property while studying the conditional structure of the exponential Brownian motion. Moreover, another paper,  Matsumoto and Yor (2003),  represents this property through hitting times of Brownian motion with drift (the case $\lambda=1/2$). 

In Proposition 2.9 of Croydon and Sasada (2020) the authors identified some  $\mathrm{GIG}$ probability measures $\mu$ and $\nu$, which satisfy the detailed balance equation \eqref{dbe}, therefore for Type I model,  with $F=F_{dK}^{(\alpha,\beta)}$ as follows: Let $c>0$ and either $\lambda\in\R$ and $\alpha\beta>0$ or $\lambda>0$ and $\alpha\beta=0$ then \bel{GIGGIG}\mu\otimes\nu=\mathrm{GIG}(-\lambda,\alpha c,c)\otimes \mathrm{GIG}(-\lambda,\beta c,c)\ee satisfies  \eqref{dbe} with $F=F_{dK}^{(\alpha,\beta)},$ with a suitable interpretation in terms of Gamma distributions of the above GIG laws when $\alpha\beta=0.$ Moreover, they observed that the characterization by the MY property given in Letac and Weso\l owski (2000) provides uniqueness of the invariant measures in case $\alpha\beta=0$. This fact was crucial for their analysis of the discrete KdV model in this special case. Searching for uniqueness of invariant measures in the generalized discrete KdV model, Croydon and Sasada (2020) conjecture our Theorem 2.2 above for $\alpha,\beta>0$, extending what was indeed already known for $\alpha\beta=0.$

\section{Proof of Theorem \ref{Th:psiAB}}
\subsection{Useful identities and extended Laplace transforms}
Observe that for $(s,t)=\psi^{(\alpha,\beta)}(a,b)$ with $a,b>0$ we have
\begin{eqnarray}\label{QUOTIENT}
&\frac{s}{t}=\frac{a}{b},\\\label{ALPHA}
&t+\alpha s=\frac{1}{a}+\frac{\beta}{b},\\\label{BETA}
&b+ \alpha a=\frac{1}{s}+\frac{\beta}{t}.
\end{eqnarray}
If $s\in\R$ and $\sigma,\theta\in (-\infty,0)$ then $(0,\infty)\ni x\mapsto x^s e^{\sigma x+\frac{\theta}{x}}$ is a bounded function. Therefore, the following functions are well defined 
\begin{eqnarray*}
	x_s(\sigma,\theta)=\E\,A^s e^{\alpha \sigma A+\frac{\theta}{A}},\quad & \quad 
	y_s(\sigma,\theta)=\E\,B^s e^{ \sigma B+\frac{\beta \theta}{B}},\\
	u_s(\sigma,\theta)=\E\,S^s e^{\alpha\theta S+\frac{\sigma}{S}},\quad &\quad 
	v_s(\sigma,\theta)=\E\,T^s e^{\theta T +\frac{\beta \sigma}{T}}, 
\end{eqnarray*}
for $(s,\sigma,\theta)\in\R\times(-\infty,0)^2$. Thus identities \eqref{QUOTIENT}, \eqref{ALPHA} and \eqref{BETA} give the equality
$$
\E\left[A^{-s}B^se^{\sigma (B+\alpha A) +\theta (\tfrac{1}{A}+\tfrac{\beta}{B})}\right]=\E\left[S^{-s}T^se^{\sigma (\tfrac{1}{S}+\tfrac{\beta}{T}) +\theta (T+\alpha S)}\right],\quad (s,\sigma,\theta)\in\R\times(-\infty,0)^2.
$$
The independence assumptions of Theorem \ref{Th:psiAB} together with the notation introduced above lead to the equation
\bel{TRICK2}
x_{-s}y_s=u_{-s}v_s,
\ee
which after taking logarithms assumes the form
\bel{TRICK3}
\log x_{-s}+\log y_s =\log u_{-s}+\log v_s.
\ee
 We now apply $\tfrac{\partial^2}{\partial \theta\partial \sigma}$ to \eqref{TRICK3}. To this aim we observe that 
\begin{eqnarray}\label{LOGPLUS}\frac{\partial^2}{\partial \theta\partial \sigma}\log v_s=\beta(1-\widetilde{v}_s),&\qquad \frac{\partial^2}{\partial \theta\partial \sigma}\log y_s\,=\beta(1-\widetilde{y}_s)
\\\label{LOGMINUS}\frac{\partial^2}{\partial \theta\partial \sigma}\log u_{-s}=\alpha(1-\widetilde{u}_{-s}),&\qquad \frac{\partial^2}{\partial \theta\partial \sigma}\log x_{-s}=\alpha(1-\widetilde{x}_{-s}),\end{eqnarray}
where we denoted
$$
 \widetilde{x}_s= \frac{x_{s-1}x_{s+1}}{x_{s}^2},\quad \widetilde{y}_s= \frac{y_{s-1}y_{s+1}}{y_{s}^2},\quad \widetilde{u}_s= \frac{u_{s-1}u_{s+1}}{u_{s}^2},\quad \widetilde{v}_s= \frac{v_{s-1}v_{s+1}}{v_{s}^2}.
$$
Consequently, \eqref{TRICK3} yields \begin{equation}\label{TATRAS1}\alpha    \widetilde{x}_{-s}+\beta \widetilde{y}_s= \alpha    \widetilde{u}_{-s}+\beta \widetilde{v}_s\end{equation}
Now observe that from \eqref{TRICK2} we have 
\begin{equation}\label{TATRAS2}\widetilde{x}_{-s}\widetilde{y}_{s}=\widetilde{u}_{-s}\widetilde{v}_{s}\end{equation}
Combining \eqref{TATRAS1} and \eqref{TATRAS2}  by eliminating either $\widetilde{x}_{-s}$ or $\widetilde{y}_{s}$ we get
\bel{TATRAS3}(\alpha \widetilde{u}_{-s}-\beta\widetilde{y}_{s})(\widetilde{v}_{s}-\widetilde{y}_{s})=(\alpha \widetilde{x}_{-s}-\beta\widetilde{v}_{s})(\widetilde{x}_{-s}-\widetilde{u}_{-s})=0.\ee

\subsection{Equality $\alpha \widetilde{u}_{-s}=\beta\widetilde{y}_{s}$ is impossible}

In this section we consider a domain $D\subset \C^3$ defined as the set of $z=(z_1,z_2,z_3)\in \C^3$ such that $\Re z_2$ and $\Re z_3$ are strictly negative.

Note that  functions $(s,\sigma,\theta)\mapsto \widetilde u_{-s}(\sigma,\theta)$ and $(s,\sigma,\theta)\mapsto \widetilde y_s(\sigma,\theta)$ are quotients of Laplace transforms on $\R^3$. These Laplace transforms are extendable to $D$ as holomorphic functions. Therefore $(s,\sigma,\theta)\mapsto f(s,\sigma,\theta)=\alpha u_{-s}(\sigma,\theta)-\beta\widetilde{y}_{s}(\sigma,\theta)$ is extendable to $D$ as a meromorphic function on $D$ (here a meromorphic function on $D$ is defined as the quotient of two holomorphic functions on $D$). Similarly, $(s,\sigma,\theta)\mapsto g(s,\sigma,\theta)=\widetilde{v}_s(\sigma,\theta)-\widetilde y_s(\sigma,\theta)$ is extendable to a meromorphic function on $D$. Furthermore, \eqref{TATRAS3} yields $f(s,\sigma,\theta)g(s,\sigma,\theta)=0$ and therefore $ f g=0$ in the field of meromorphic functions on $D.$ 
This implies that either $f(s,\sigma,\theta)=0$ for all $(s,\sigma,\theta)$ or $g(s,\sigma,\theta)=0$ for all $(s,\sigma,\theta).$ 

It is worthwhile to recall the short proof of the fact that if $f(z)g(z)=0$ in $D$ then either $f(z)=0$ for all $z\in D$ or $g(z)=0$ for all $z\in D.$ We thank A. Zeriahi for it. Indeed if $f=f_1/f_2,\ g=g_1/g_2$ where $f_1,f_2,g_1,g_2$ are holomorphic in $D$ with $f_2$ and  $g_2$ not identically zero, we must have $f_1(z)g_1(z)=0.$ 
We use the notations $k!=k_1!k_2!k_3!$ and  $$g_1^{(k)}(z)=\tfrac{\partial^{k_1+k_2+k_3} }{\partial z_1^{k_1}\partial z_2^{k_2}\partial z_3^{k_3}}g_1(z_1,z_2,z_3),\ \  (z-a)^{(k)}=(z_1-a_1)^{k_1}(z_2-a_2)^{k_2}(z_3-a_3)^{k_3}.$$  
Suppose now that there exists $z_0\in D$ such that $f_1(z_0)\neq 0$ and  consider the set 
$$D_0=\left\{z\in D; g_1^{(k)}(z)=0\,  \mathrm{for \ all}\, k\in \N^3\right\}.$$  The set $D_0$ is an open subset of $D$, since for all $a\in D_0$ the  set $$\left\{z\in D; g_1(z)=\sum_{k\in \N^3}(z-a)^{(k)}\tfrac{g_1^{(k)}(a)}{k!}\right\}$$ is a subset of $D_0$ containing a neighborhood of $a.$  Also $D_0$ is a closed subset of $D$ and $D_0$ is not empty since it contains $z_0$. Since $D$ is connected we have $D=D_0. $ The same reasoning will show that if $g(z_0)\neq 0$ then $f\equiv 0.$

Next we show that $\alpha \widetilde{u}_{-s}=\beta\widetilde{y}_{s}$
is impossible. Indeed, in that case from \eqref{LOGPLUS} and \eqref{LOGMINUS} for $s=0$ we get
$$
\alpha-\tfrac{\partial^2}{\partial\theta\partial\sigma}\,\log\, u_0=\beta-\tfrac{\partial^2}{\partial\theta\partial\sigma}\,\log\, y_0,
$$
which gives
\begin{align}
\label{eEeE}
e^{\beta\theta\sigma+H(\theta)}\,\E\,e^{\sigma B+\tfrac{\beta\theta}{B}}&=e^{\alpha\theta\sigma+G(\sigma)}\,\E\,e^{\alpha\theta S+\tfrac{\sigma}{S}},\quad (\sigma,\theta)\in(-\infty,0)^2,
\end{align}
for some functions $G$ and $H$. Plugging  $(\theta,0)$, $(0,\sigma)$ and $(0,0)$ in \eqref{eEeE} we get
\begin{align}
\label{G0}
e^{G(0)}\E\,e^{\alpha\theta S}&=e^{H(\theta)}\E\,e^{\tfrac{\beta\theta}{B}},\\
\label{H0}
e^{G(\sigma)}\,\E\,e^{\tfrac{\sigma}{S}}&=e^{H(0)}\,\E\,e^{\sigma B},\\
\label{G0H0}
e^{H(0)}&=e^{G(0)}.
\end{align}
Multiplying \eqref{eEeE}-\eqref{G0H0} side-wise and taking $\sigma=\theta$ we arrive at
\begin{equation}
\label{betasigma}
e^{\beta\theta^2}\,\E\,e^{\theta \left(B+\tfrac{\beta}{B}\right)}\,\E\,e^{\theta\alpha S}\,\E\,e^{\tfrac{\theta}{S}}=e^{\alpha\theta^2}\,\E\,e^{\theta\left(\alpha S+\tfrac{1}{S}\right)}\,\E\,e^{\theta\tfrac{\beta}{B}}\,\E\,e^{\theta B}.
\end{equation} 
Now introduce $S_1\stackrel{d}{=}S_2\stackrel{d}{=}S$ such that $S_1,S_2$ and $B$  are independent, and $B_1\stackrel{d}{=}B_2\stackrel{d}{=}B$ such that $B_1,B_2$ and $S$  are independent. Consider positive random variables $P_1=B+\tfrac{\beta}{B}+\alpha S_1+\tfrac{1}{S_2}$ and $P_2=\alpha S+\tfrac{1}{S}+\tfrac{\beta}{B_1}+B_2$. From \eqref{betasigma} we get
$$
e^{\beta\theta^2}\,\E\,e^{\theta P_1}=e^{\alpha\theta^2}\,\E\,e^{\theta P_2}.
$$
Now assume $\alpha<\beta$ and consider a Gaussian random variable $Z\sim N(0,2(\alpha-\beta))$ such that $Z$ and $P_2$ are independent. Then we get $\E\,e^{\theta P_1}=\E\,e^{\theta (P_1+Z}$, $\theta<0$, which implies $P_1\stackrel{d}{=}P_2+Z$, which is a contradiction since the support of the distribution of $P_2+Z$ is $\R$. A similar contradiction is obtained for $\beta<\alpha$.

\subsection{Exploiting the connection between  $v_s$ and $y_s$}

So, we necessarily have 
\begin{equation}\label{vvv}\widetilde v_s=\widetilde y_s.\end{equation}
Combining this equality with \eqref{LOGPLUS}  we obtain for all $s\in\R$ that
$$\frac{\partial^2}{\partial \theta\partial \sigma}(\log y_s-\log v_s)=0.$$ Thus, there exist functions $ H_s$ and $G_{s}$ satisfying
\begin{equation}\label{YV}
\log v_s(\sigma,\theta)\,=\,H_s(\theta)-G_s(\sigma)+\log y_s(\sigma,\theta).
\end{equation}
Since $v_s$ and $y_s$ are analytic functions of $\sigma$ and $\theta$ it follows that $G_s$  and $H_s$ are also analytic. 
Thus we can take derivatives of \eqref{YV} with respect to $\theta$ and $\sigma;$ whence we get \begin{eqnarray}\label {YVSIGMA}
%\frac{\partial}{\partial \sigma}\log v_s=
\beta\frac{v_{s-1}(\sigma,\theta)}{v_s(\sigma,\theta)}&=&-G_s'(\sigma)+ \frac{y_{s+1}(\sigma,\theta)}{y_s(\sigma,\theta)}\\
\label {YVTHETA}
%\frac{\partial}{\partial \theta}\log v_s=
\frac{v_{s+1}(\sigma,\theta)}{v_s(\sigma,\theta)}&=&H_s'(\theta)+\beta \frac{y_{s-1}(\sigma,\theta)}{y_s(\theta,\sigma)}.\end{eqnarray}
Side-wise multiplication of  equations \eqref{YVSIGMA} and \eqref{YVTHETA} yields 
$$\beta\widetilde{v}_s(\theta,\sigma)=-H_s'(\theta)G_s'(\sigma)-\beta G_s'(\sigma)\frac{y_{s-1}(\sigma,\theta)}{y_s(\sigma,\theta)}+H_s'(\theta)\frac{y_{s+1}(\sigma,\theta)}{y_s(\sigma,\theta)}+\beta \widetilde{y}_s(\sigma,\theta),$$ 
which, in view of \eqref{vvv} yields
\begin{equation}
\label{H'G'}
H_s'(\theta)G_s'(\sigma)=H_s'(\theta)\tfrac{y_{s+1}(\sigma,\theta)}{y_s(\sigma,\theta)}-\beta G_s'(\sigma)\tfrac{y_{s-1}(\sigma,\theta)}{y_s(\sigma,\theta)}.
%=G_s'(\sigma)\tfrac{v_{s+1}(\theta,\sigma)}{v_s(\theta,\sigma)}-\beta\,H_s'(\theta)\,\tfrac{v_{s-1}(\theta,\sigma)}{v_s(\theta,\sigma)}.
\end{equation}
The fact that $G'_s$ and $H'_s$ can be zero or not is crucial in the sequel. Observe that in view of \eqref{H'G'}  if $G'_s(\sigma_1)=0$ for some $\sigma_1<0$ then $H'_s(\theta)=0$ for every $\theta<0$. By symmetry we conclude that for any $s\in\R$ either $G'_s$ and $H'_s$ are identically zero or they are both non-zero for any value of arguments. Therefore $$\Lambda:=\{s\in\R:\,G'_s\equiv 0\equiv H'_s\}=\{s\in\R:\,G'_sH'_s\equiv 0\}.$$

\subsection{Properties of the set $\Lambda$}
\begin{proposition}\label{rs}
	\begin{enumerate}
		\item [(i)] If $s\in\Lambda$ then there exists a constant $C(s)>0$ such that $v_s=C(s)y_s$.
		\item[(ii)] There is no $s\in\R$ such that $s,s-1\in\Lambda$. 
		%For any $s\in\R$ we have $\{s-1,s\}\cap (\R\setminus\Lambda)\neq\emptyset$.
		\item[(iii)] There is no $s\in\R$ such that $s\not\in\Lambda$ and $s-1,s+1\in\Lambda$.
		\item[(iv)] There exists $s\in\R$ such that $s,s-1\not\in\Lambda$.
	\end{enumerate}
\end{proposition}
\begin{proof}
	\begin{enumerate}
		\item [Ad(i)] Follows immediately for \eqref{YV}.
		\item[Ad(ii)] From \eqref{YVTHETA}, in view of (i) we have
		$$
		\tfrac{\beta C(s-1)y_{s-1}}{C(s)y_s}=\tfrac{y_{s+1}}{y_s}
		$$
		which implies $y_{s+1}=ky_{s-1}$, where $k=\tfrac{\beta C(s-1)}{C(s)}$. This gives $\tfrac{\partial y_s}{\partial \theta}=k\,\tfrac{\partial y_s}{\partial \sigma}$. 
		Consequently, there exists a function $f$ defined on $(-\infty,0)$ such that $y_s(\sigma,\theta)=f(k\theta+\sigma)$. Denoting $t=k\theta+\sigma$ we get for all $t<0$ and all $\theta\in\left(\tfrac{t}{k},\,0\right)$
		$$
		y_s(\sigma,\theta)=\E\,B^s\,\exp\left(tB+\theta\left(\tfrac{\beta}{B}-kB\right)\right)=f(t).
		$$
		Therefore for $t$ fixed $\theta\mapsto y_s(\theta,t-k\theta)$ is constant on $\left(\tfrac{t}{k},\,0\right)$  and thus we conclude that $\tfrac{\beta}{B}-kB$ is constant. Since $B>0$ and non Dirac we get a contradiction.
		\item[Ad(iii)] Using \eqref{vvv} and part (i) we get $v_s^2=C(s-1)C(s+1)y_s^2$. Since $s\not\in\Lambda$ it follows from \eqref{YV} that $v_s=e^{H_s(\theta)-G_s(\sigma)}\,y_s$, which contradicts the fact that for $s\not\in\Lambda$ the function $H_s(\theta)-G_s(\sigma)$ is not constant.  
		\item[Ad(iv)] It follows from (ii) that there exists $s\not\in\Lambda$. If $s-1,s+1\in\Lambda$ it contradicts (iii). Therefore at least one of $s\pm 1$ is not in $\Lambda$. 
	\end{enumerate}
\end{proof}

\begin{proposition}
	\begin{enumerate}
		\item[(i)] If $s\not\in\Lambda$ then there exist $r_s\neq 0$, $\sigma_0(s)\ge 0$, $\theta_0(s)\ge 0$ such that for  $\theta<0$ and $\sigma<0$ we have
		\begin{equation}
		\label{DERIV}
		G'_s(\sigma)=\tfrac{r_s}{\sigma_0(s)-\sigma}\quad \mbox{and}\quad H'(\theta)=\tfrac{r_s}{\theta_0(s)-\theta}.
		\end{equation}
		Furthermore, there exist constants $C_G(s)$ and $C_H(s)$ such that for all $\theta<0$ and $\sigma<0$
		\begin{equation}\label{GH}
		G_s(\sigma)=-r_s\log(\sigma_0(s)-\sigma)+C_G(s),\quad\mbox{and}\quad H_s(\theta)=-r_s\log(\theta_0(s)-\theta)+C_H(s).
		\end{equation}
		\item[(ii)] If $s,s-1\not\in\Lambda$ then $r_s=r_{s-1}+1$, $\sigma_0(s-1)=\sigma_0(s)$, $\theta_0(s-1)=\theta_0(s)$.
	\end{enumerate}
\end{proposition}
\begin{proof}\begin{enumerate}
	\item[Ad(i)]
	Since $s\not\in\Lambda$ we can divide by $H'_sG'_s$ in \eqref{H'G'} getting
	\begin{eqnarray}\label{XYUV1}1&=&\tfrac{1}{G'_s(\sigma)}\tfrac{y_{s+1}(\sigma,\theta)}{y_s(\sigma,\theta)}-\tfrac{\beta}{H'_s(\theta)}\tfrac{y_{s-1}(\sigma,\theta)}{y_s(\sigma,\theta)}\\\label{XYUV2}1&=&\frac{1}{H'_s(\theta)}\tfrac{v_{s+1}(\sigma,\theta)}{v_s(\sigma,\theta)}-\tfrac{\beta}{G'_s(\sigma)}\frac{v_{s-1}(\sigma,\theta)}{v_s(\sigma,\theta)}.
	%\\\label{XYUV3}-1&=&\frac{\alpha}{H'}\frac{u_{-s-1}}{u_{-s}}-\frac{1}{G'_s}\frac{u_{-s+1}}{u_{-s}}\\\label{XYUV4}1&=&\frac{1}{H'_s}\frac{x_{-s-1}}{x_{-s}}-\frac{\alpha}{G'}\frac{x_{-s+1}}{x_{-s}}
	\end{eqnarray} 
	
	 Comparing \eqref{YVSIGMA} and \eqref{XYUV1} we obtain
	
	\begin{equation}\label{THEBEST1}\frac{1}{G'_s(\sigma)}\frac{v_{s-1}(\sigma,\theta)}{v_s(\sigma,\theta)}=\frac{1}{H'_s(\theta)}\frac{y_{s-1}(\sigma,\theta)}{y_s(\sigma,\theta)}.\end{equation}
	We take derivative of \eqref{THEBEST1} with respect to $\sigma$
	
	\begin{equation}\label{THEBEST2}\frac{v_{s-1}(\sigma,\theta)}{v_s(\sigma,\theta)}\frac{G_s''(\sigma)}{G'^2_s(\sigma)}-\frac{1}{G'_s(\sigma)}\,\frac{\partial}{\partial \sigma}\left(\frac{v_{s-1}(\sigma,\theta)}{v_s(\sigma,\theta)}\right)+\frac{1}{H'_s(\theta)}(1-\widetilde{y}_s(\theta,\sigma))=0.\end{equation}
	Since
	$$\frac{\partial}{\partial \sigma}\left(\frac{v_{s-1}}{v_s}\right)=-\beta\left(\frac{v_{s-1}}{v_s}\right)^2(1-\widetilde{v}_{s-1}),$$
	using this last equality and \eqref{THEBEST1}, after cancelling by 
	$\frac{v_{s-1}}{v_s}$, we get 
	
	$$
	\frac{G''_s(\sigma)}{(G'_s(\sigma))^2}+\frac{1}{G_s'(\sigma)}\left(\beta\frac{v_{s-1}(\sigma,\theta)}{v_s(\sigma,\theta)}-\frac{y_{s+1}(\sigma,\theta)}{y_s(\sigma,\theta)}\right)-\frac{1}{G'_s(\sigma)}\left(\beta\frac{v_{s-2}(\sigma,\theta)}{v_{s-1}(\sigma,\theta)}-\frac{y_{s}(\sigma,\theta)}{y_{s-1}(\sigma,\theta)}\right)=0$$
	Finally, using \eqref{YVSIGMA} we obtain 
	
	\begin{equation} \label{THEBEST3}\frac{G_s''}{G_s'^2}-\frac{G'_s-G'_{s-1}}{G_s'}=0\end{equation}
	A similar calculation starting with comparison \eqref{YVTHETA} and \eqref{XYUV2}, and taking the derivative with respect to $\theta$ gives the same thing with $H$:
	\begin{equation} \label{THEBEST4}\frac{H_s''}{H_s'^2}-\frac{H'_s-H'_{s-1}}{H_s'}=0\end{equation} Now reactivating  \eqref{THEBEST1} we get
	and 
	$$\frac{G'_s(\sigma)}{H'_s(\theta)}=\frac{v_{s-1}(\sigma,\theta)}{v_s(\sigma,\theta)}\frac{y_s(\sigma,\theta)}{y_{s-1}(\sigma,\theta)}\quad \mbox{and}\quad
	\frac{H'_{s-1}(\theta)}{G'_{s-1}(\sigma)}=\frac{v_{s-1}(\sigma,\theta)}{v_{s-2}(\sigma,\theta)}\frac{y_{s-2}(\sigma,\theta)}{y_{s-1}(\sigma,\theta)}.
	$$
	Consequently, 
	$$
	\frac{G'_s(\sigma)}{H'_s(\theta)}\frac{H'_{s-1}(\theta)}{G'_{s-1}(\sigma)}=\frac{\widetilde{y}_{s-1}(\sigma,\theta)}{\widetilde{v}_{s-1}(\sigma,\theta)}=1.$$ Combining this result with \eqref{THEBEST3} and \eqref{THEBEST4} we land on a separation of variables equation
	$$\frac{G_s''(\sigma)}{G_s'^2(\sigma)}=\frac{H_s''(\theta)}{H_s'^2(\theta)}=\frac{1}{r_s}$$ where $r_s$ is a non-zero constant with respect to $\sigma$ and $\theta.$ Integrating these two simple differential equations we arrive at \eqref{DERIV}.
	
	The formula \eqref{GH} is an immediate consequence of \eqref{DERIV}.
	
	\item[Ad(ii)]  From \eqref{THEBEST3}, in view of \eqref{DERIV}, we get 
	$$
	\tfrac{1}{r_s}=1-\tfrac{r_{s-1}}{r_s}\,\tfrac{\sigma_0(s)-\sigma}{\sigma_0(s-1)-\sigma},\quad \sigma<0.
	$$
	Therefore,
	$$
	(r_s-1-r_{s-1})\sigma+r_{s-1}\sigma_0(s)-(r_s-1)\sigma_0(s-1)=0.
	$$
	Thus, looking at the coefficient of $\sigma$ we get $r_s=r_{s-1}+1$. Thus, the constant term gives $\sigma_0(s)=\sigma_0(s-1)$. 	Similarly, using \eqref{THEBEST4} we get $\theta_0(s)=\theta_0(s-1)$.
	\end{enumerate}
	
\end{proof}

%\begin{remark}
%	\label{not necessary} Although it is not necessary in the proof we can mention that $s\in\Lambda$ implies that $s\pm n\not\in\Lambda$ for any positive integer $n$. For $n=1$ this is given in Proposition \ref{rs} (iii). Consider the case $n\ge 2$. Assume $s,s+n\in\Lambda$ and $s+1,\ldots,s+n-1\not\in\Lambda$.  and use \eqref{uxvy} 
%\end{remark}

\subsection {Computing $y_s$ and $v_s$} 

 Let us introduce an auxiliary function
 \begin{equation}
 \label{yyGH}
 \overline{y}_s(\sigma,\theta)=\log y_s(\sigma,\theta)-\d G_s(\sigma)+\d H_s(\theta).
 \end{equation}
Thus, in view of \eqref{GH} we get
 \begin{equation}
 \label{EyyGH}y_s(\sigma,\theta)=C(s)e^{\overline y_s(\sigma,\theta)}\left(\tfrac{\theta_0-\theta}{\sigma_0-\sigma}\right)^{r_s/2}.
 \end{equation}
 Rewrite now \eqref{XYUV1} as
\begin{equation}
\label{1-G'}
1-\frac{1}{G'_s(\sigma)}\frac{\partial \log y_s(\sigma,\theta)}{\partial \sigma}+\frac{1}{H'_s(\theta)}\frac{\partial \log y_s(\sigma,\theta)}{\partial \theta}=0.
\end{equation}
Note that \eqref{EyyGH} yields
$$
\frac{\partial \log y_s(\sigma,\theta)}{\partial \sigma}=\frac{\partial  \overline y_s(\sigma,\theta)}{\partial \sigma}+\tfrac{r_s}{2(\sigma_0-\sigma)}\quad\mbox{and}\quad \frac{\partial \log y_s(\sigma,\theta)}{\partial \theta}=\frac{\partial  \overline y_s(\sigma,\theta)}{\partial \theta}-\tfrac{r_s}{2(\theta_0-\theta)}.
$$
Plugging these two expressions into \eqref{1-G'}, in view of \eqref{DERIV},   we get
\begin{equation}
\label{pde}
(\theta_0-\theta)\frac{\partial \overline{y}_s}{\partial \theta}-(\sigma_0-\sigma)\frac{\partial \overline{y}_s}{\partial \sigma}=0.
\end{equation}
For 
\begin{equation}
\label{pq}
p=\sqrt{(\sigma_0-\sigma)(\theta_0-\theta)}\quad\mbox{ and }\quad q=\sqrt{\sigma_0-\sigma}/\sqrt{\theta_0-\theta},
\end{equation}
we introduce temporarily a function $h_s$ defined by $h_s(p,q)=\overline y_s(\theta,\,\sigma)$.  Since $$h_s(p,q)=\overline y_s\left(\theta_0-\tfrac{p}{q} ,\,\sigma_0-pq\right)$$ we arrive at
$$\frac{\partial h_s(p,q)}{\partial q}=\frac{p}{q^2}\frac{\partial \overline{y}_s}{\partial \theta}-p\frac{\partial \overline{y}_s}{\partial \sigma}=\frac{1}{q}\left(  (\theta_0-\theta)\frac{\partial \overline{y}_s}{\partial \theta}-(\sigma_0-\sigma)\frac{\partial \overline{y}_s}{\partial \sigma}\right).$$ Thus, \eqref{pde} implies that $h_s(p,q)=\ell_s(2p)$ for some function $\ell_s$. Consequently, 
\begin{equation}\label{OUF}\overline y_s(\sigma,\theta)=\ell_s\left(2\sqrt{(\sigma_0-\sigma)(\theta_0-\theta)}\right).\end{equation}
Denote $f_s=C(s)e^{\ell_s}$ with $C(s)$ from \eqref{EyyGH}. Referring to \eqref{OUF} and \eqref{EyyGH} we get
\begin{equation}
\label{ysfs}
y_s(\sigma,\theta)=\left(\frac{\theta_0-\theta}{\sigma_0-\sigma}\right)^{r_s/2}
f_s\left(2\sqrt{(\theta_0-\theta )(\sigma_0-\sigma)}\right)=q^{-r_s}f_s(2p).
\end{equation}
To identify $f_s$ we will derive an ordinary second order differential equation satisfied by this function.  We first show that
\begin{equation}\label{DIFFATLAST}
\beta y_s(\sigma,\theta)-\frac{r_{s-1}}{\sigma_0-\sigma}\frac{\partial y_s(\sigma,\theta)}{\partial \theta}-\frac{\theta_0-\theta}{\sigma_0-\sigma}\frac{\partial ^2y_s(\sigma,\theta)}{\partial \theta^2}=0.\end{equation}
To this end we use \eqref{DERIV} and rewrite \eqref{XYUV1}  with $s$ changed into $s-1$ as
$$
r_{s-1}y_{s-1}(\sigma,\theta)=(\sigma_0-\sigma)\,y_s(\sigma,\theta)-\beta(\theta_0-\theta)y_{s-2}(\sigma,\theta).
$$
Now, \eqref{DIFFATLAST} follows from the last equality since $\tfrac{\partial y_s}{\partial \theta}=\beta y_{s-1}$ and $\tfrac{\partial^2 y_s}{\partial \theta^2}=\beta^2 y_{s-2}.
$
Then we insert $y_s$ as obtained in \eqref{ysfs} into \eqref{DIFFATLAST}. Recalling \eqref{pq} we note that
\begin{equation}
\label{pqd}
\tfrac1{q}\,\tfrac{\partial q}{\partial \theta}=\tfrac{1}{2(\theta_0-\theta)}\quad\mbox{and}\quad \tfrac1{p}\,\tfrac{\partial p}{\partial \theta}=-\tfrac{1}{2(\theta_0-\theta)}.
\end{equation}
Using \eqref{pqd}, after careful calculation, we get
$$
\tfrac{\partial y_s(\sigma,\theta)}{\partial \theta}=-\tfrac{q^{-r_s}}{\theta_0-\theta}\left[\tfrac{r_s}{2}f_s(2p)+f_s'(2p)\right]
$$ 
and
$$
\tfrac{\partial^2 y_s(\sigma,\theta)}{\partial \theta^2}=\tfrac{q^{-r_s}}{(\theta_0-\theta)^2}\left[-\tfrac{r_s}{2}\left(1-\tfrac{r_s}{2}\right)f_s(2p)+\left(r_s-\tfrac12\right)f_s'(2p)+p^2f''_s\right].
$$ 
Plugging these two expressions into \eqref{DIFFATLAST} we get
$$
f_s''(2p)-\left(r_{s-1}-r_s+\tfrac12\right)\,\tfrac{f_s'(2p)}{p}-\left(\beta p^2+\tfrac{r_{s-1}r_s}{2}+\tfrac{r_s}{2}-\tfrac{r_s^2}{4}\right)\tfrac{f_s(2p)}{p^2}.
$$
Using the fact that $r_{s-1}=r_s-1$ the above equation can be rewritten as
$$
f_s''(2p)+\tfrac{f_s'}{2p}-\left(\beta p^2+\tfrac{r_s^2}{4}\right)\,\tfrac{f(2p)}{p^2}=0.
$$
Denoting $t=2p$ we get
$$
f_s''(t)+\tfrac{f_s'(t)}{t}-(\beta t^2+r_s^2)\tfrac{f_s(t)}{t^2}=0.
$$
Thus, for a function $g_s$ defined by $g_s(z)=f_s(z/\sqrt{\beta})$ we obtain the classical Bessel equation (see e.g. Watson (1966), p.77, (1))
$$
g_s''(z)+\tfrac{g_s'(z)}{z}-(z^2+r_s^2)\tfrac{g_s(z)}{z^2}=0.
$$
Consequently, $g_s=aI_{r_s}+bK_{r_s}$, where $I_{r_s}$ and $K_{r_s}$ are the modified Bessel functions of the first and third type, defined by
$$
I_{\nu}(z)=\sum_{m=0}^{\infty}\,\tfrac{(z/2)^{\nu+2m}}{m!\Gamma(\nu+m+1)}, 
$$
e.g.  Watson (1966), p.77 (2), and
\begin{equation}
\label{besse}
K_{\nu}(z)=\left(\tfrac{z}{2}\right)^\nu\,\int_0^{\infty}\,\tau^{-\nu-1}\,e^{-(\tau+\tfrac{z^2}{4\tau})}\,d\tau,
\end{equation}
e.g.  Watson (1966), p. 183 (15).
We note that $g_s$ is  bounded at infinity. Indeed, it suffices to show that $f_s$ is bounded at infinity. Note that for a fixed $s$ the function $y_s$ is bounded on $(-\infty,-\eps)^2$ for a fixed $\eps>0$. Inserting $\theta_0-\theta=\sigma_0-\sigma=t/2$ in \eqref{ysfs}  for  $(\sigma,\theta)\in(-\infty,-\eps)^2$  we get $y_s(\theta,\sigma)=f_s(t)$. Thus $f_s$ is bounded at infinity.
But $I_\nu(z)\rightarrow_{z\to\infty}\infty$. Therefore, $g_s=bK_{r_s}$ for some real $b$ depending on $s$. Consequently, for $\sigma<0$ and $\theta<0$ we have
\begin{equation}
\label{PAINY}
y_s(\sigma,\theta)=q^{-r_s}g_s(2\sqrt{\beta}p)=b(s)\left(\tfrac{\theta_0-\theta}{\sigma_0-\sigma}\right)^{r_s/2}\,K_{r_s}\left(2\sqrt{\beta(\theta_0-\theta)(\sigma_0-\sigma)}\right).
\end{equation}
Repeating for $v_s$ the argument we used for $y_s$ while using \eqref{XYUV2} instead of \eqref{XYUV1} requires only flipping $\theta_0-\theta$ with $\sigma_0-\sigma$. Therefore,
\begin{equation}
\label{PAINV}
v_s(\sigma,\theta)=t(s)\left(\tfrac{\sigma_0-\sigma}{\theta_0-\theta}\right)^{r_s/2}\,K_{r_s}\left(2\sqrt{\beta(\theta_0-\theta)(\sigma_0-\sigma)}\right).
\end{equation}
Note that in both formulas \eqref{PAINY} and \eqref{PAINV} we have $\theta_0\ge 0$ and $\sigma_0\ge 0$.

\subsection{Conclusion of the proof}

Let us first show that $\theta_0>0$ and $\sigma_0>0.$ For this consider a possibly unbounded positive measure
$$
r_B(dx)=x^{r_s-s-1}\,e^{-\sigma_0 x-\tfrac{\beta\theta_0}{x}}\,I_{(0,\infty)}(x)\,dx
$$
Then for $\sigma<0$ and $\theta<0$ we have
$$
\int_0^{\infty}\,x^se^{\sigma s+\tfrac{\beta\theta_0}{s}}\,r_{B}(dx)=\left(\tfrac{\beta(\theta_0-\theta)}{\sigma_0-\sigma}\right)^{r_s/2}\,K_{r_s}(2\sqrt{\beta(\sigma_0-\sigma)(\theta_0-\theta)})
= \tfrac{\beta^{r_s/2}}{b(s)}y_s(\sigma,\theta),
$$
where the last equality follows by \eqref{PAINY}. Consequently, we get
$$
\int_0^{\infty}\,x^se^{\sigma s+\tfrac{\beta\theta_0}{s}}\,r_{B}(dx)=\tfrac{\beta^{r_s/2}}{b(s)}\,\int_0^{\infty}\,\,x^s\,e^{\sigma x+\tfrac{\beta\theta}{x}}\,\P_B(dx).
$$
Therefore,
\begin{equation}
\label{PB}
\P_B(dx)\propto x^{r_s-s-1}\,e^{-\sigma_0 x-\tfrac{\beta\theta_0}{x}}\,I_{(0,\infty)}(x)\,dx
\end{equation}
Similarly, using \eqref{PAINV}, we see that 
\begin{equation}
\label{PT}
\P_T(dx)\propto x^{r_s-s-1}\,e^{-\theta_0 x-\tfrac{\beta\theta_0}{x}}\,I_{(0,\infty)}(x)\,dx.
\end{equation}
Note that \eqref{PB} or \eqref{PT} imply that $\theta_0=\sigma_0=0$ is impossible. Note also that when $\theta_0>0$ and $\sigma_0=0$, then, according to \eqref{PB}, the measure $\P_B$ is finite only if $r_s-s<0$. But then the measure at the right hand side of \eqref{PT} is not finite, which is impossible. By the similar argument we exclude the case $\sigma_0>0$ and $\theta_0=0$. Thus, both these situations are not possible. Finally, we conclude that $\sigma_0>0$, $\theta_0>0$ and $B\sim\mathrm{GIG}(\lambda,\sigma_0,\beta\theta_0)$, $T\sim\mathrm{GIG}(\lambda,\theta_0,\beta\sigma_0)$, with $\lambda=r_s-s$.

To compute the distributions of $A$ and $S$ we follow the arguments we used for $B$ and $T$, but we start with the second equality in \eqref{TATRAS3}, instead of the first one, which we used while computing $v_s$ and $y_s$. This equality yields $\widetilde u_{-s}=\widetilde x_{-s}$ for all $\sigma<0$ and $\theta<0$. As in the case $\widetilde v_s=\widetilde y_s$ we obtain  $A\sim\mathrm{GIG}(\lambda',\,\alpha\sigma'_0,\,\theta'_0)$ and $S\sim\mathrm{GIG}(\lambda',\alpha\theta'_0,\sigma'_0)$ for some $\sigma_0',\theta_0'>0$ and $\lambda'\in\R$. 

Consequently, there exist constants $C_i$, $i=1,2,3,4$, which may depend on $s$ such that
\begin{align*}
&x_{-s}=C_1\left(\tfrac{\theta_0'-\theta}{\sigma_0'-\sigma}\right)^{(\lambda'-s)/2}\,K_{\lambda'-s}\left(2\sqrt{\alpha(\sigma'_0-\sigma)(\theta'_0-\theta)}\right),\\
&y_{s}=C_2\left(\tfrac{\theta_0-\theta}{\sigma_0-\sigma}\right)^{(\lambda+s)/2}\,K_{\lambda+s}\left(2\sqrt{\beta(\sigma_0-\sigma)(\theta_0-\theta)}\right),\\
&u_{-s}=C_3\left(\tfrac{\sigma_0'-\sigma}{\theta_0'-\theta}\right)^{(\lambda'-s)/2}\,K_{\lambda'-s}\left(2\sqrt{\alpha(\sigma'_0-\sigma)(\theta'_0-\theta)}\right),\\
&v_{s}=C_4\left(\tfrac{\sigma_0-\sigma}{\theta_0-\theta}\right)^{(\lambda+s)/2}\,K_{\lambda+s}\left(2\sqrt{\beta(\sigma_0-\sigma)(\theta_0-\theta)}\right).
\end{align*}
Thus \eqref{TRICK2} implies
$$C_1C_2\,\left(\tfrac{\theta_0'-\theta}{\sigma_0'-\sigma}\right)^{(\lambda'-s)/2}\,\left(\tfrac{\theta_0-\theta}{\sigma_0-\sigma}\right)^{(\lambda+s)/2}=C_3C_4\,\left(\tfrac{\sigma_0'-\sigma}{\theta_0'-\theta}\right)^{(\lambda'-s)/2}\,\left(\tfrac{\sigma_0-\sigma}{\theta_0-\theta}\right)^{(\lambda+s)/2},
$$
for all $\sigma<0$, $\theta<0$. Consequently, $\lambda+s=-(\lambda'-s)$, i.e. $\lambda=-\lambda'$. Moreover, $\theta_0=\theta'_0:=c_2$ and $\sigma_0=\sigma'_0:=c_1$.  Thus the result follows.

\section{GMY property for GIG matrices}
Let $\Omega$ be the Euclidean space of symmetric $r\times r$ matrices with real entries and the inner product defined by $\<x,y\>=\tr(xy)$, $x,y\in \Omega$. Denote by $\Omega_+$ the cone of symmetric positive definite $r\times r$ matrices. We say that an $\Omega_+$-valued random matrix $X$ has a matrix GIG law, $\mathrm{MGIG}(p,a,b)$, for $p\in\R$ and $a,b\in\Omega_+$ if it has the density with respect to the Lebesgue measure $dx$ in $\Omega$ of the form
$$
\mu_{p,a,b}(x)=\tfrac{1}{K_p(a,b)}\,(\det\,x)^{p-\tfrac{r+1}{2}}\,\exp\left(-\tfrac{\tr(ax)+\tr(bx^{-1})}{2}\right)\,I_{\Omega_+}(x),
$$
see e.g. Sec. 2 of Letac and Weso\l owski (2000).

For $\alpha,\beta\ge 0$ we recall the map $F_{dK}^{(\alpha,\beta)}$ on $\Omega_+\times\Omega_+$ defined in the introduction by the formula
\bel{MF}
F_{dK}^{(\alpha,\beta)}(x,y)=\left(y(I+\alpha xy)^{-1}(I+\beta xy),\,x(I+\beta yx)^{-1}(I+\alpha yx)\right),
\ee
where $I$ denotes the identity matrix. Note that for $\alpha=0$, $\beta>0$ and $\alpha>0$, $\beta=0$ the problem we discuss reduces to the matrix variate version of the original Matsumoto-Yor property given in Letac and Weso\l owski (2000).  Moreover, $F_{dK}^{(\alpha,\alpha)}$ is just the identity. Therefore, we are not interested in these cases in the sequel.

\begin{proposition}
	For every distinct $\alpha,\beta>0$ the map $F_{dK}^{(\alpha,\beta)}$ is a differentiable involution on $\Omega_+\times\Omega_+$. Moreover, its Jacobian equals 1.
\end{proposition}
\begin{proof} For  distinct $\alpha,\beta>0$ we denote $F:=F_{dK}^{(\alpha,\beta)}$. For $x,y\in\Omega_+$ set $(u,v)=F(x,y)$. Note that
	$$y(I+\alpha xy)^{-1}=(I+ \alpha yx)^{-1}y\qquad\mbox{and}\qquad y(I+\beta xy)=(I+\beta yx)y.
	$$
	Combining these two identities we get
	$$
	u=y(I+\alpha xy)^{-1}(I+\beta xy)=(I+\alpha yx)^{-1}y(I+\beta xy)=(I+\alpha yx)^{-1}(I+\beta yx)y.
	$$
	But the right hand side above equals $u^T$, i.e. $u$ is symmetric. Similarly, we conclude that $v$ is symmetric. 
	Since $u$ and $v$ are products of positive definite matrices, they are also positive definite. 
	Consequently, $F(\Omega_+\times\Omega_+)\subset \Omega_+\times\Omega_+$.
	Note that $v^T=(I+\alpha xy)(I+\beta xy)^{-1}x=(I+\beta xy)^{-1}(I+\alpha xy)x$. Since $v=v^T$ we get
	\bel{uvyx}
	uv=yx.
	\ee
	Plugging \eqref{uvyx} or its transpose to definitions of $u$ and $v$ we get
	$$
	u=y(I+\alpha vu)^{-1}(I+\beta vu)\qquad\mbox{and}\qquad v=x(I+\beta uv)^{-1}(I+\alpha uv)
	$$
	whence
	$$
	x=v(I+\alpha uv)^{-1}(I+\beta uv)\qquad\mbox{and}\qquad y=u(I+\beta vu)^{-1}(I+\alpha vu).
	$$
	Thus $F$ is an involution, i.e. $F=F^{-1}$. Differentiability on $\Omega_+\times\Omega_+$ is clear.
	
	To compute the Jacobian $J_F$ of $F$ we first note that $F=\phi^{-1}\circ\psi\circ \phi$, where $\phi(x,y)=(x,y^{-1})$ and 
	\bel{psi}\psi(x,y)=\left(y^{-1}(y+\beta x)(y+\alpha x)^{-1},\;x^{-1}(y+\beta x)(y+\alpha x)^{-1}\right)=:(u',v').\ee
	Thus $$J_F(x,y)=J_{\phi^{-1}}(x,y)J_{\psi}(x,y^{-1})J_{\phi}(u',v').$$
	Note that
	$$u'=\tfrac{\beta}{\alpha}\,y^{-1}+\tfrac{\alpha-\beta}{\alpha}(y+\alpha x)^{-1}\qquad \mbox{and}\qquad v'=x^{-1}-(\alpha-\beta)(y+\alpha x)^{-1}.$$
	Recall that the derivative  $D_x(x^{-1})$ is $-\P_{x^{-1}}$, where $\P_a$ is the  endomorphism on $\Omega$ defined by $\P_a(h)=aha$, $h\in\Omega$. Consequently, the derivatives $D_x$ and $D_y$ of  $u'$  are
	$$
	D_x(u')=-(\alpha-\beta)\P_{(y+\alpha x)^{-1}}\qquad \mbox{and}\qquad D_y(u')=-\tfrac{\beta}{\alpha}\P_{y^{-1}}-\tfrac{\alpha-\beta}{\alpha}\,\P_{(y+\alpha x)^{-1}}.
	$$
	Similarly, the derivatives  of  $v'$ are
	$$
	D_x(v')=-\P_{x^{-1}}+(\alpha-\beta)\alpha\P_{(y+\alpha x)^{-1}}\qquad \mbox{and}\qquad D_y(v')=(\alpha-\beta)\P_{(y+\alpha x)^{-1}}.
	$$
	For clarity we write $\det x$ for the determinant of $x\in\Omega$ and $\mathrm{Det}\,M$ for the determinant of the endomorphism $M$ on $\Omega$ or $\Omega^2$. For example, in the sequel we will need the formula 
	\begin{equation}
	\label{Det}
	\mathrm{Det}\,\P_x=(\det\,x)^{r+1},
	\end{equation} see e.g. Faraut and Koranyi (1994), p. 52. Thus
	\begin{align*}
	J_{\psi}(x,y)&=\mathrm{Det}\,\left[\begin{array}{cc} -(\alpha-\beta)\P_{(y+\alpha x)^{-1}} & -\tfrac{\beta}{\alpha}\P_{y^{-1}}-\tfrac{\alpha-\beta}{\alpha}\,\P_{(y+\alpha x)^{-1}}\\
	\, & \, \\ 
	-\P_{x^{-1}}+(\alpha-\beta)\alpha\P_{(y+\alpha x)^{-1}} & (\alpha-\beta)\P_{(y+\alpha x)^{-1}}\end{array}\right]\\\, & \, \\
	&=\mathrm{Det}\,\left[\begin{array}{cc} -(\alpha-\beta)\P_{(y+\alpha x)^{-1}} \ \ & \ \ -\tfrac{\beta}{\alpha}\P_{y^{-1}}-\tfrac{\alpha-\beta}{\alpha}\,\P_{(y+\alpha x)^{-1}}\\ \, & \, \\
	-\P_{x^{-1}} \ \ & \ \ -\beta\P_{y^{-1}}\end{array}\right].
	\end{align*}
	Using the Cholesky decomposition we thus get
	\begin{align*}
	J_{\psi}(x,y)=&\mathrm{Det}\,\left\{(\alpha-\beta)\P_{(y+\alpha x)^{-1}}\right\}\\
	&\times \mathrm{Det}\left\{-\beta\P_{y^{-1}}+\tfrac{1}{\alpha(\alpha-\beta)}\,\P_{x^{-1}}\left(\P_{(y+\alpha x)^{-1}}\right)^{-1} \left(\beta\P_{y^{-1}}+(\alpha-\beta)\,\P_{(y+\alpha x)^{-1}}\right)\right\}.
	\end{align*}
	Using the fact that $\P_{a^{-1}}=\P_a^{-1}$ we see that the operator under the second determinant has the form
	\begin{align*}
	&-\beta\P_{y^{-1}}+\tfrac{1}{\alpha(\alpha-\beta)}\,\P_{x^{-1}}\P_{y+\alpha x} \left(\beta\P_{y^{-1}}+(\alpha-\beta)\,\P_{y+\alpha x}^{-1}\right)\\
	=&-\beta\P_{y^{-1}}+\tfrac{\beta}{\alpha(\alpha-\beta)}\,\P_{x^{-1}}\P_{y+\alpha x} \P_{y^{-1}}+\tfrac{1}{\alpha}\P_{x^{-1}}\\
	=&\tfrac{1}{\alpha(\alpha-\beta)}\,\P_{x^{-1}}\left[-\alpha\beta(\alpha-\beta)\P_x+\beta\,\P_{y+\alpha x} +(\alpha-\beta)\P_y\right]\P_{y^{-1}}
	\end{align*}
	For $x,y\in\Omega$ denote by $\mathbb L_{x,y}$ the endomorphism of $\Omega$ defined by $\mathbb L_{x,y}(h)=xhy+yhx$, $h\in\Omega$. Then
	$$
	\P_{y+\alpha x}=\P_y+\alpha \mathbb L_{x,y}+\alpha^2\P_x
	$$
	whence
	$$
	-\alpha\beta(\alpha-\beta)\P_x+\beta\,\P_{y+\alpha x} +(\alpha-\beta)\P_y=-\alpha\beta(\alpha-\beta)\P_x+\beta(\alpha^2\P_x+\alpha \mathbb L_{x,y}+\P_y)+(\alpha-\beta)\P_y=
	\alpha\P_{y+\beta x}.
	$$
	Summing up, in view of \eqref{Det}, we obtain
	$$
	J_{\psi}(x,y)=\tfrac{\mathrm{Det}\,\P_{y+\beta x}}{\mathrm{Det}\,\P_{y+\alpha x}\;\mathrm{Det}\,\P_x\;\mathrm{Det}\,\P_y}=\left(\tfrac{\det(y+\beta x)}{\det(y+\alpha x)\det(x)\det(y)}\right)^{r+1}.
	$$
	Thus 
	$$
	J_{\psi}(x,y^{-1})=\left(\tfrac{\det(I+\beta xy)\det(y)}{\det(I+\alpha xy)\det(x)}\right)^{r+1}
	$$
	Note that $J_{\phi}(x,y)=\mathrm{Det}\left\{-\P_{y^{-1}}\right\}=\tfrac{(-1)^r}{(\det\,y)^{r+1}}$. Therefore,
	$$
	J_F(x,y)=\tfrac{1}{(\det\,y)^{r+1}}\,\left(\tfrac{\det(I+\beta xy)\det(y)}{\det(I+\alpha xy)\det(x)}\right)^{r+1}\tfrac{1}{(\det(v'(x,y^{-1})))^{r+1}}.
	$$
	From \eqref{psi} we have $\det\,v'(x,y^{-1})=\tfrac{\det(I+\beta xy)}{\det(x)\,\det(I+\alpha xy)}$. Thus we conclude that $\left|J_F(x,y)\right|=1$.
\end{proof}

Now we are ready to formulate the main result of this section which gives a detailed balance equation of the generalized Matsumoto-Yor type satisfied by GIG random matrices.
\begin{theorem}
	Let $(X,Y)\sim \mathrm{MGIG}(\lambda,\alpha a,b)\otimes\mathrm{MGIG}(\lambda,\beta b,a)$.
	Then
	$$(U,V)=F_{dK}^{(\alpha,\beta)}(X,Y)\sim \mathrm{MGIG}(\lambda,\alpha b,a)\otimes\mathrm{MGIG}(\lambda,\beta a,b).$$
\end{theorem}

\begin{proof}
	Consider the joint density $f$ of $(U,V)$. Since the Jacobian of $F_{dK}^{(\alpha,\beta)}$ is 1. It follows that
	\begin{align*}
	f(u,v)&=f_X(x(u,v))f_Y(y(u,v))\\
	&\propto (\det(xy))^{\lambda-\tfrac{r+1}{2}}\,\exp\left(-\tfrac{\alpha\tr(ax)+\tr(bx^{-1})+\beta\tr(by)+\tr(ay^{-1})}{2}\right)
	\end{align*}
	We now consider the exponent. Denoting $\widetilde\alpha=I+\alpha uv\in\Omega_+$, $\widetilde\beta=I+\beta uv\in\Omega_+$ and using the fact that $\tr\,a=\tr\,a^T$ we get
	\begin{align*}
	&\alpha\tr(ax)+\tr\left(bx^{-1}\right)+\beta\tr(by)+\tr\left(ay^{-1}\right)\\
	=&\alpha\tr\left(av\widetilde{\alpha}^{-1}\widetilde{\beta}\right)+\tr\left(b\widetilde{\beta}^{-1}\widetilde{\alpha}v^{-1}\right)+\beta\tr\left(bu\left(\widetilde{\beta}^T\right)^{-1}\widetilde\alpha^T\right)+\tr\left(a\left(\widetilde\alpha^T\right)^{-1}\widetilde{\beta}^Tu^{-1}\right).
	\end{align*}
	Since 
	$$u\left(\widetilde{\beta}^T\right)^{-1}\widetilde\alpha^T=\widetilde\beta^{-1}\widetilde\alpha u\qquad\mbox{and}\qquad \left(\widetilde\alpha^T\right)^{-1}\widetilde{\beta}^Tu^{-1}=u^{-1}\widetilde\alpha^{-1}\widetilde{\beta}$$
	We get
	\begin{align*}
	&\alpha\tr(ax)+\tr\left(bx^{-1}\right)+\beta\tr(by)+\tr\left(ay^{-1}\right)\\
	=&\tr\left(a\left(u^{-1}+\alpha v\right)\widetilde\alpha^{-1}\widetilde{\beta}\right)+\tr\left(b\widetilde{\beta}^{-1}\widetilde{\alpha}\left(v^{-1}+\beta u\right)\right)\\
	=&\tr\left(au^{-1}(I+\alpha uv)\widetilde\alpha^{-1}\widetilde{\beta}\right)+\tr\left(b\widetilde{\beta}^{-1}\widetilde{\alpha}(I+\beta uv)v^{-1}\right)\\
	=&\tr\left(au^{-1}(I+\beta uv)\right)+\tr\left(b(I+\alpha uv)v^{-1}\right)\\
	=&\tr\left(au^{-1}\right)+\tr(\beta a v)+\tr\left(bv^{-1}\right)+\tr(\alpha b u).
	\end{align*}
	Combining this with \eqref{uvyx} we get
	$$
	f(u,v)\propto (\det(u))^{\lambda-\tfrac{r+1}{2}}\exp\left(-\tfrac{\tr\left(\alpha b u\right)+\tr\left(au^{-1}\right)}{2}\right)\,(\det(v))^{\lambda-\tfrac{r+1}{2}}\,\exp\left(-\tfrac{\tr(\beta a v)+\tr\left(bv^{-1}\right)}{2}\right),
	$$
	which ends the proof.
\end{proof}

\section{Bibliographical comments}
Let us mention that the MY property and  related characterization triggered a lot of further research developing in several directions: (1) more general algebraic structures as,  a multivariate tree-generated version in Massam and Weso\l owski (2004), matrix variate versions in Letac and Weso\l owski (2000), Weso\l owski (2002), Massam and Weso\l owski (2006), a combination of the matrix variate and multivariate tree-generated setting in Bobecka (2015), symmetric cone variate in Ko\l odziejek (2017), a version in free probability in Szpojankowski (2017); (2) characterizations based on a weaker assumption of constancy of regressions of moments of $T$ given $S$ instead of the assumption of independence of $S$ and $T$ - in univariate case in Weso\l owski (2002), Chou and Wang (2004) and in free probability in Szpojankowski (2017) and \'Swieca (2021); (3) a search of more general maps of the form $\psi_f(a,b)=(f(a+b),\,f(a)-f(a+b))$ and product measures $\mu\otimes\nu$, such that $\psi_f(\mu\otimes\nu)$ remains a product measure  and characterization of respective $\mu$ and $\nu$ by the independence property in Koudou and Vallois (2012) with a genuine special case of $\mu$ and $\nu$ being the Kummer and gamma distributions, see also Koudou and Vallois (2011),  Koudou (2011), Weso\l owski (2015),  Piliszek and Weso\l owski (2016), Ko\l odziejek (2018). For a survey on characterizations of the GIG distribution and other refrences (up to 2014) see Koudou and Ley (2014).

\small 
{\bf References}
\begin{enumerate}
	\item {\sc Bao, K.V., Noack, C.}, Characterizations of the generalized inverse Gaussian, asymmetric Laplace, and shifted (truncated) exponential laws via independence properties. {\bf 2107.01394} (2021), 1--12.
	
	\item {\sc Bobecka, K.}, The Matsumoto-Yor property on trees for matrix variates of different dimensions. {\em J.Multivar. Anal.} {\bf 141} (2015), 22–-34.
	
	\item  {\sc Chou, C.-W.,  Huang, W.-J.}, On characterizations of the gamma and generalized inverse Gaussian distributions, {\em Statist. Probab. Lett.} {\bf 69} (2004), 381-–388.
	
	\item {\sc Croydon, D.A., Sasada, M.}, Detailed balance and invariant measures for systems of locally-defined  dynamics. {\em arXiv} {\bf 2007.06203} (2020), 1--48.
	
	\item {\sc Croydon, D.A., Sasada, M., Tsujimoto, S.}, General solutions for KdV- and Toda-type discrete integrable systems based on path encodings. {\em arXiv} {\bf 2011.00690} (2020), 1--81.
	
	\item {\sc Faraut, J., Koranyi, A.}, {\em Analysis on Symmetric Cones}, Cambridge University Press (1994).
	
\item {\sc Ko\l odziejek, B.}, Matsumoto-Yor property and its converse on symmetric cones. {\em J. Theoret. Probab.} {\bf 30} (2017), 624--638.
	
	\item {\sc Ko\l odziejek, B.},  A Matsumoto-Yor characterization for Kummer and Wishart random matrices. {\em J. Math. Anal. Appl.} {\bf 460(2)} (2018), 976--986.
	
	\item {\sc Koudou, A.E.}, A Matsumoto-Yor property for Kummer and Wishart matrices, {\em Statist. Probab. Lett.} {\bf 82(11)} (2012), 1903-–1907.
	
	\item {\sc Koudou, A.E., Ley, C.}, Characterizations of GIG laws: a survey. {\em Probab. Surveys} {\bf 11} (2014), 161-176. 
	
	\item {\sc Koudou, A.E.,  Vallois, P.}, Which distributions have the Matsumoto-Yor property? {\em Electron. Commun. Probab.} {\bf  16} (2011), 556-–566.
	
	\item {\sc Koudou, A.E.,  Vallois, P.}, Independence properties of the Matsumoto-Yor type. {\em Bernoulli} {\bf 18(1)} (2012), 119–-136.
	
	\item {\sc Letac, G., Weso\l owski, J.},  An independence property for the GIG and gamma laws {\em Ann. Probab.} {\bf 28(3)} (2000), 1371–-1383.
	
	\item {\sc Massam, H., Weso\l owski, J.},  The Matsumoto-Yor property on trees. {\em Bernoulli} {\bf 10} (2004), 685–-700.
	
	\item {\sc Massam, H., Weso\l owski, J.}, The Matsumoto-Yor property and the structure of the Wishart distribution. {\em J. Multivar. Anal.} 97 (2006), 103–-123.
	
	\item {\sc Matsumoto, H., Yor, M.},  An analogue of Pitman's 2M-X theorem  for exponential Wiener functionals. Part II: the role of the generalized inverse Gaussian laws {\em Nagoya Math. J.} {\bf 162} (2001), 65-–86.
	
    \item {\sc Matsumoto, H., Yor, M.}, Interpretation via Brownian motion of some independence properties between GIG and gamma variables. {\em Statist. Probab. Lett.} {\bf 61} (2003), 253--259.
	
	\item {\sc Piliszek, A., Weso\l owski, J.},  Kummer and gamma laws through independencies on trees –- another parallel with the Matsumoto–Yor property. {\em  J.Multivar. Anal.} {\bf 152} (2016), 15-–27.
	
	\item {\sc Stirzaker, D.}, {\em Stochastic Processes \& Models}, Oxford Univ. Press, Oxford 2005.
	
	\item {\sc Szpojankowski, K.}, On the Matsumoto-Yor property in free probability. {\em J. Math. Anal. Appl.} {\bf 445(1)} (2017), 374--393.
	
	\item {\sc \'Swieca, M.}, The Matsumoto-Yor property in free probability via subordination and Boolean cummulants. {\em arXiv} {\bf 2109.12545} (2021), 1--24.
	
	\item {\sc Weso\l owski, J.} The Matsumoto-Yor independence property for GIG and gamma laws, revisited. {\em Math. Proc. Cambridge Philos. Soc.} 133 (2002), 153–-161.
	
	\item {\sc Weso\l owski, J.} On the Matsumoto-Yor type regression characterization of the gamma and Kummer distributions. {\em Statist. Probab. Lett.} {\bf 107} (2015), 145-–149.
\end{enumerate}
\end{document}